\documentclass[11pt]{article}

\usepackage[T1]{fontenc}
\usepackage{lmodern}

\usepackage{amsmath, amssymb, amsthm} 
\usepackage[margin=2.5cm]{geometry}
\usepackage{color, graphicx}
\usepackage[caption=false]{subfig}

\usepackage{etoolbox}
\makeatletter
\patchcmd{\@maketitle}{\LARGE \@title}{\LARGE\bfseries\@title}{}{}

\renewcommand{\@seccntformat}[1]{\csname the#1\endcsname.\quad}
\makeatother

\definecolor{darkblue}{rgb}{0,0,.5}

\usepackage{hyperref}
\hypersetup{
	colorlinks=true,		
	linkcolor=darkblue,		
	citecolor=red,		
	urlcolor=darkblue 		
}

\makeatletter
\def\th@plain{%
	\thm@notefont{}
	\itshape 
}
\def\th@definition{%
	\thm@notefont{}
	\normalfont 
}

\renewenvironment{proof}[1][\proofname]{\par
	\normalfont
	\topsep0\p@\@plus3\p@ \trivlist
	\item[\hskip\labelsep\itshape
	#1\@addpunct{.}]\ignorespaces
}{%
	\qed\endtrivlist
}
\makeatother

\newtheorem{theorem}{Theorem}[section]
\newtheorem{lemma}[theorem]{Lemma}

\theoremstyle{definition}
\newtheorem{definition}[theorem]{Definition}
\theoremstyle{definition}
\newtheorem{example}[theorem]{Example}
\theoremstyle{definition}
\newtheorem{remark}[theorem]{Remark}

\usepackage[shortlabels]{enumitem}

\allowdisplaybreaks

\newcommand{\R}{\ensuremath{\mathbb R}}
\newcommand{\ball}{\ensuremath{\mathbb B}}

\newcommand{\ran}{\ensuremath{\operatorname{ran}}}
\newcommand{\zer}{\ensuremath{\operatorname{zer}}}
\newcommand{\dom}{\ensuremath{\operatorname{dom}}}

\newcommand{\gra}{\ensuremath{\operatorname{gra}}}
\newcommand{\Fix}{\ensuremath{\operatorname{Fix}}}

\def\beq{\begin{equation}}
\def\eeq{\end{equation}}
\def\baq{\begin{eqnarray}}
\def\eaq{\end{eqnarray}}
\def\baqn{\begin{eqnarray*}}
\def\eaqn{\end{eqnarray*}}

\if
{
\usepackage[pageref]{backref}
\renewcommand*{\backrefalt}[4]{%
\ifcase #1 %
(Not cited)%
\or
(Cited on p.~#2)%
\else
(Cited on pp.~#2)%
\fi
}

}
\fi

\usepackage{float}

\begin{document}

\title{Splitting Algorithms Using  Pairs Monotonicity to Solve the Sum of Two Non-monotone Operators Problem}

\author{
Ba Khiet Le\thanks{Analytical and Algebraic Methods in Optimization Research Group, Faculty of Mathematics and Statistics, Ton
Duc Thang University, Ho Chi Minh City, Vietnam. E-mail: \texttt{lebakhiet@tdtu.edu.vn}}, 
~
Zakaria Mazgouri\thanks{Sidi Mohammed Ben Abdellah University, 
National School of Applied Sciences, Fez, Morocco. 
E-mail: \texttt{zakariam511@gmail.com}}, 
~and~ 
Michel Th\' era\thanks{Mathematics and Computer Science Department, University of Limoges, 123 Avenue Albert Thomas,
87060 Limoges CEDEX, France. E-mail: \texttt{michel.thera@unilim.fr}}
}


\maketitle

\begin{abstract}
\noindent {This paper is concerned with solving the problem of finding a zero of the sum of a single-valued operator and a set-valued operator, neither of which is assumed to be monotone, in a real Hilbert space. Under the framework of pair monotonicity, we propose a generalized forward--backward algorithm and a generalized Tseng  algorithm based on transformed and warped resolvents associated with an auxiliary linear operator. We establish weak, strong and linear convergence of the proposed algorithms under suitable assumptions involving pair strong monotonicity, pair cocoercivity, and pair Lipschitz continuity. Finally, numerical experiments are presented to validate the theoretical results and demonstrate the effectiveness of the proposed algorithms.
} 
\end{abstract}


\paragraph{Mathematics Subject Classification (MSC 2020):}
 49J52, 49J53.

\section{Introduction}
In this paper, we focus on the inclusion problem
\begin{equation}\label{main-pb}
0\in A(x)+B(x),
\end{equation}
where $\mathcal{H}$ is a real Hilbert space, $A:\mathcal{H}\to\mathcal{H}$ is single-valued, and $B:\mathcal{H}\rightrightarrows\mathcal{H}$ is set-valued. This fundamental formulation arises in a broad range of problems in optimization, variational analysis, equilibrium theory, and related areas. In particular, when $A$ and $B$ are maximally monotone, \eqref{main-pb} provides the underlying framework for numerous splitting and proximal methods, and has been extensively investigated from both analytical and computational perspectives; see, e.g., \cite{Apidopoulos,Attouch,Bauschke,Bauschke1,Chen,cp,cw,Davis,Gibali,Giselsson,Giselsson1,He,Hong,Lions,Moursi,Svaiter,Tseng}.

A substantially more challenging setting arises when monotonicity is absent. Many important models, particularly those originating in nonconvex optimization and equilibrium problems, lead naturally to inclusions in which neither operator is monotone. Our objective is to address precisely this setting by studying \eqref{main-pb} without imposing monotonicity assumptions on either $A$ or $B$. This departure from the classical framework is not merely technical: the loss of monotonicity fundamentally alters the geometry of the problem and, in general, rules out the direct use of standard splitting and proximal techniques. It therefore calls for new structural conditions that can recover, at least partially, the monotonicity mechanisms underlying the convergence of these methods.

A promising framework in this direction is provided by the notion of \emph{pair monotonicity}, originally introduced in \cite{acl} and subsequently developed in \cite{LDT,LMT}. Rather than requiring an operator to be monotone in its own right, pair monotonicity seeks to identify an auxiliary mapping that interacts with the given operator so as to restore a generalized monotonicity structure. This viewpoint is particularly well suited to nonmonotone inclusions, since it allows one to exploit a hidden monotonicity mechanism without imposing restrictive monotonicity assumptions on the original operator.

More precisely, pair monotonicity was used in \cite{acl} to study a broad class of quasi-variational inequalities (QVIs) and inclusions of the form
\begin{equation}\label{1oper}
0\in\mathcal{A}(x),
\end{equation}
where $\mathcal{A}:\mathcal{H}\rightrightarrows\mathcal{H}$ is not necessarily monotone. The key idea is to construct an auxiliary mapping $v:\mathcal{H}\to\mathcal{H}$ such that the pair $(\mathcal{A},v)$ satisfies an appropriate generalized monotonicity relation. In this framework, $v$ acts as a structural companion to $\mathcal{A}$ and is chosen to reflect the particular analytical properties of the problem. Thus, rather than modifying the original operator or imposing global monotonicity, one seeks to reveal a suitable monotonicity structure through the interaction between $\mathcal{A}$ and $v$.

This perspective leads naturally to algorithmic developments. When $\mathcal{A}$ is single-valued and $v$ is invertible, the pair-monotonicity framework yields a forward-type algorithm for solving \eqref{1oper}; see \cite{acl}. For set-valued operators, this approach was extended in \cite{LDT} through a generalized proximal point algorithm ($\mathbf{GPPA})$, thereby providing a mechanism for treating nonmonotone inclusions within a proximal framework. These results demonstrate that pair monotonicity can serve as a unifying structural principle for designing convergent algorithms beyond the classical monotone setting.

Problem \eqref{1oper} arises naturally in unconstrained optimization. Consider
$$
\min_{x\in\mathcal{H}} f(x),
$$
where $f:\mathcal{H}\to\mathbb{R}\cup\{+\infty\}$ is not necessarily convex. Under appropriate assumptions, a first-order necessary optimality condition can be expressed as
$$
0\in\partial f(x),
$$
where $\partial f$ denotes an appropriate generalized subdifferential, such as the Mordukhovich or Clarke subdifferential. A natural extension is the constrained optimization problem
$$
\min_{x\in C} f(x)
=\min_{x\in\mathcal{H}}\bigl(f(x)+I_C(x)\bigr),
$$
where $C\subset\mathcal{H}$ is a convex set and $I_C$ denotes its indicator function. Under suitable qualification conditions, a necessary optimality condition takes the form
\begin{equation}\label{constr}
0\in\partial f(x)+N_C(x),
\end{equation}
where $N_C$ denotes the normal cone mapping. Although $N_C$ is maximally monotone, the generalized subdifferential $\partial f$ is generally nonmonotone when $f$ is nonconvex. Nevertheless, the pair-monotonicity framework offers a way to recover a useful monotonicity structure: it may be possible to identify an auxiliary mapping $v$ such that both pairs $(\partial f,v)$ and $(N_C,v)$ are monotone. This observation provides a direct connection between pair monotonicity and constrained nonconvex optimization.

More generally, consider the composite minimization problem
$$
\min_{x\in\mathcal{H}}\bigl(f(x)+g(x)\bigr),
$$
where $f,g:\mathcal{H}\to\mathbb{R}\cup\{+\infty\}$. Under appropriate qualification conditions, its first-order optimality condition leads naturally to the inclusion \eqref{main-pb}, with $A=\partial f$ and $B=\partial g$. Thus, both constrained optimization and composite nonconvex optimization fit naturally within the framework of nonmonotone sum inclusions. This connection motivates our study of \eqref{constr} and, more generally, of \eqref{main-pb} through pair monotonicity.

Existing developments based on pair monotonicity have primarily focused on general nonmonotone inclusions and generalized proximal point methods; see \cite{LDT,LMT}. The present work takes a complementary step by explicitly exploiting the sum structure in \eqref{main-pb}. Rather than applying a proximal method to the inclusion as a whole, we develop splitting schemes that treat the single-valued operator $A$ by a forward step and the set-valued operator $B$ through a generalized resolvent. This distinction is particularly important for structured applications in which $A$ is explicitly computable, while $B$ possesses a resolvent that can be evaluated efficiently, as is the case for many normal cone and subdifferential operators. Our approach therefore extends the pair-monotonicity framework from generalized proximal methods to genuinely splitting-based algorithms for nonmonotone sum inclusions, while retaining the computational advantages associated with the separate treatment of $A$ and $B$.

Our first contribution is a generalized forward--backward algorithm based on a transformed resolvent of $B$ ($\mathbf{Algorithm~1}$). We establish linear convergence under two alternative pair-monotonicity conditions: either $(A,v)$ or $(B,v)$ is strongly monotone; see Theorems~\ref{thm1} and~\ref{thm2}, respectively. We further establish strong convergence under the assumption that {$(A,v)$ or $(v,A)$} is pair cocoercive  see Theorem~\ref{thm3}. These results provide convergence guarantees for a forward--backward scheme even though neither of the original operators is required to be monotone.

When strong monotonicity or cocoercivity is absent, the generalized forward–backward framework may no longer be directly applicable. To overcome this limitation, we introduce a generalized Tseng forward–backward–forward algorithm ($\mathbf{Algorithm 2}$) for computing approximate solutions. In particular, we allow $v$ to be nonsurjective, which provides greater flexibility in the choice of $v$ and thereby significantly broadens the range of potential applications.
Moreover, under the additional assumption that $v^{-1}$ is Lipschitz continuous, we prove weak convergence of the generated sequence; see Theorem \ref{thm4}. 

Taken together, these results demonstrate that pair monotonicity can be effectively combined with classical splitting principles to handle a substantially broader class of nonmonotone inclusions. In particular, the proposed framework separates the structural role of pair monotonicity from the computational roles of the two operators, thereby allowing convergence analysis and algorithm design to proceed without requiring either $A$ or $B$ to be monotone individually.

The remainder of the paper is organized as follows. Section \ref{sec2} reviews the notation and basic concepts related to pair monotonicity that will be used throughout the paper. In Section \ref{sec3}, we introduce the generalized forward--backward and generalized Tseng algorithms and establish their convergence properties. Section \ref{sec4} presents numerical experiments illustrating the theoretical results. Finally, Section \ref{sec5} concludes the paper with a summary of the main findings and directions for future research.

\section{{Preliminaries }}\label{sec2}

{
Throughout this paper, \( \mathcal{H} \) denotes a real Hilbert space endowed with the inner product \( \langle \cdot, \cdot \rangle \) and the induced norm \( \|\cdot\| \).  For
a set-valued operator \( F : \mathcal{H} \rightrightarrows \mathcal{H} \), the \emph{domain}, \emph{range}, \emph{graph}, \emph{set of zeros}, and \emph{set of fixed points} of \( F \) are defined, respectively, by }
\begin{align*}
\dom F &= \{x \in \mathcal{H} : F(x) \neq \varnothing\}, \\
\ran F &= \bigcup_{x \in \mathcal{H}} F(x), \\
\gra F &= \{(x, y) : x \in \mathcal{H}, y \in F(x)\}, \\
\zer F &= \{x \in \mathcal{H} : 0 \in F(x)\}, \\
\Fix F &= \{x \in \mathcal{H} : x \in F(x)\}.
\end{align*}  

The \emph{inverse} of \( F \) is defined by \( F^{-1}(y) = \{x \in \mathcal{H} : y \in F(x)\} \). It is easy to see that \( \dom F = \ran F^{-1} \) and \( \ran F = \dom F^{-1} \).  

The set-valued mapping \( F \) is said to be \emph{monotone} if for all \( (x, x^*), (y, y^*) \in \gra F \),
\[
\langle x^* - y^*, x - y \rangle \geq 0,
\]
and \emph{\(\alpha\)-strongly monotone} if \( \alpha \in (0, +\infty) \) and for all \( (x, x^*), (y, y^*) \in \gra F \),
\[
\langle x^* - y^*, x - y \rangle \geq \alpha \|x - y\|^2.
\] 

	Next, we  recall the concept of \emph{pair monotonicity}, introduced in \cite{acl}, which provides a natural extension of the classical notion of monotonicity, which is a useful tool in the analysis of non-monotone inclusion problems (see \cite{LDT,LMT}). Let \( F_1, F_2 : \mathcal{H} \rightrightarrows \mathcal{H} \) be two set-valued mappings, we define 

\begin{itemize}
\item
The pair \( (F_1, F_2) \) is \emph{monotone} if for all \( x, y \in \mathcal{H} \),
$$
\langle F_1(x) - F_1(y), F_2(x) - F_2(y) \rangle \geq 0,
$$
i.e., for all \( x, y \in \mathcal{H} \), \( x_1^* \in F_1(x) \), \( y_1^* \in F_1(y) \), \( x_2^* \in F_2(x) \), \( y_2^* \in F_2(y) \),
$$
\langle x_1^* - y_1^*, x_2^* - y_2^* \rangle \geq 0.
$$
\item
The pair \( (F_1, F_2) \) is \emph{\(\alpha\)-strongly monotone} if \( \alpha \in (0, +\infty) \) and for all \( x, y \in \mathcal{H} \),
$$
\langle F_1(x) - F_1(y), F_2(x) - F_2(y) \rangle \geq \alpha \|x - y\|^2,
$$
i.e., for all \( x, y \in \mathcal{H} \), \( x_1^* \in F_1(x) \), \( y_1^* \in F_1(y) \), \( x_2^* \in F_2(x) \), \( y_2^* \in F_2(y) \),
$$
\langle x_1^* - y_1^*, x_2^* - y_2^* \rangle \geq \alpha \|x - y\|^2.
$$
\end{itemize}  

\begin{remark}
While classical monotonicity compares an operator with the identity mapping ($F_2=Id$), pair monotonicity replaces the identity by a suitable auxiliary mapping, allowing certain non-monotone operators to inherit monotonicity-like properties.   	
\end{remark}  

\begin{definition}
The pair $(A,v)$ is called Lipschitz continuous if there exists $L>0$ such that 
$$
\Vert Ax-Ay\Vert\le L \Vert v(x)-v(y)\Vert.
$$
\end{definition}

\begin{definition}
The pair $(A,v)$ is called co-coercive  if there exists $\alpha>0$ such that 
$$
\langle Ax-Ay, v(x) - v(y) \rangle \geq \alpha \|Ax - Ay\|^2.
$$
\end{definition}

For a set-valued operator \( F : \mathcal{H} \rightrightarrows \mathcal{H} \) and a linear $v: \mathcal{H} \to \mathcal{H}$ with \( \dom v = \mathcal{H} \), we recall the following definitions. The \emph{transformed resolvent} of $F$ with respect to $v$ \cite{LDT} is defined by     
\[
T_{F}^v = v \circ (F + v)^{-1},
\] 
while the \emph{warped resolvent} of $F$ with kernel $v$ \cite{bc} is defined by
\[
J_F^v = (F + v)^{-1} \circ v,
\] 
provided that \( \ran v \subseteq \ran(F + v) \). These generalized resolvents will serve as the main tools in the design of the algorithms introduced in the next section.

We   now  recall the notion of \emph{$R$-continuity} for a set-valued mapping
$\mathcal{A}:\mathcal{H}\rightrightarrows\mathcal{H}$ at a point
$\bar{x}$ in its domain; see \cite{Le,LMT,LT}. This property is useful in
the convergence analysis of the algorithms developed below.

\begin{definition}\label{rdef}
The set-valued mapping \noindent $\mathcal{A}:\mathcal{H} \rightrightarrows \mathcal{H}$ is called {\sc R-continuous} at ${\bar x}\in\dom\,\mathcal{A}$ if there exist a number $\sigma>0$ and a nondecreasing function $\rho: \mathbb{R}^+\to \mathbb{R}^+$ satisfying $\lim_{r\to 0^+}\rho(r)=\rho(0)=0$ such that we have the inclusion
\begin{equation}\label{rcon}
\mathcal{A}(x) \subset \mathcal{A}({\bar x}  )+\rho(\Vert x-{\bar x}   \Vert)\ball\;\mbox{ for all }\;x\in  \ball({\bar x} ,\sigma)
\end{equation}
with the {\sc continuity modulus function} $\rho$ and {\sc radius} $\sigma$. When $\sigma=\infty$, $\mathcal{A}$ is  called {\sc globally R-continuous} at ${\bar x}$. The inclusion in \eqref{rcon} means that for each $y\in \mathcal{A}(x)$, there exists ${\bar y}  \in  \mathcal{A}({\bar x} )$ satisfying the estimate 
$$
\Vert y-{\bar y}  \Vert\le \rho(\Vert x-{\bar x}   \Vert)
$$ 
for all $x\in \ball({\bar x},\sigma)$.
\end{definition}

%

\begin{definition}\label{closed graph}
Let $\mathcal{A}:\mathcal{H}\rightrightarrows\mathcal{H}$ be a set-valued
mapping. We say that:

\begin{itemize}
    \item the graph of $\mathcal{A}$ is  \emph{sequentially   closed }at a  point 
    $\bar{x}\in\dom\mathcal{A}$ if, for any sequences
    $x_k\to\bar{x}$ and $y_k\in\mathcal{A}(x_k)$ with $y_k\to y$, one has
    $y\in\mathcal{A}(\bar{x})$;

    \item the graph of $\mathcal{A}$ is \emph{sequentially weak-to-strong
    closed} if, for any sequences $x_k\rightharpoonup x$ and
    $y_k\in\mathcal{A}(x_k)$ with $y_k\to y$, one has
    $y\in\mathcal{A}(x)$.
\end{itemize}
\end{definition}

\begin{theorem}\label{compact}
\cite{LMT1,LT}
Suppose that $\mathcal{A}:\mathcal{H}\rightrightarrows\mathcal{H}$ has a
closed graph at $0$ and is locally compact around $0$, in the sense that
there exists $\sigma>0$ such that
\[
\mathcal{A}\bigl(\sigma\mathbb{B}\setminus\{0\}\bigr)
\]
is contained in a compact subset of $\mathcal{H}$. Then $\mathcal{A}$ is
$R$-continuous at $0$.
\end{theorem}

We finally recall the classical weak convergence lemma of Opial, which will be used to establish the weak convergence of the sequence generated by Algorithm~2.

\begin{lemma}[Opial's Lemma {\cite{Opial}}]\label{lem:Opial}
Let $S$ be a nonempty subset of $\mathcal{H}$, and let $(x_n)_{n\in\mathbb{N}}$ be a sequence in $\mathcal{H}$. Suppose that:
\begin{enumerate}[(i)]
    \item for every $x^*\in S$, the limit
    \[
    \lim_{n\to\infty}\|x_n-x^*\|
    \]
    exists;
     \item every sequential weak cluster point of $(x_n)_{n\in\mathbb{N}}$ belongs to $S$.
\end{enumerate}
Then $(x_n)_{n\in\mathbb{N}}$ converges weakly to some point $x^\infty\in S$.
\end{lemma}

\section{Main results}\label{sec3}

In this section, we study the convergence analysis of  algorithms that we propose to solve (\ref{main-pb}).   Throughout this section, we assume that the solution set
\[
S := \{x\in \mathcal{H} : 0\in A(x)+B(x)\},
\]
of \eqref{main-pb} is nonempty.

\medskip

\noindent\textbf{Assumption 1.} There exists a  mapping $v:\mathcal{H}\to\mathcal{H}$ such that the pairs $(A,v)$ and $(B,v)$ are monotone.

\medskip

\noindent\textbf{Assumption 2.} There exists $\gamma>0$ such that
\[
\operatorname{ran}(\gamma B+v)=\mathcal{H}.
\]

\medskip

Under Assumption 2, the transformed resolvent $T_{\gamma B}^{v}$ is well defined. Moreover, if the pair $(B,v)$ is monotone, then $T_{\gamma B}^{v}$ is firmly nonexpansive, i.e.,
\[
\|T_{\gamma B}^{v}(u)-T_{\gamma B}^{v}(w)\|^2 \le \langle T_{\gamma B}^{v}(u)-T_{\gamma B}^{v}(v),u-w\rangle,
\qquad \forall,u,w\in\mathcal{H}.
\]
Consequently, $T_{\gamma B}^{v}$ is  nonexpansive, i.e.,
\[
\|T_{\gamma B}^{v}(u)-T_{\gamma B}^{v}(w)\|^2 \le \|u-w\|,
\qquad \forall,u,w\in\mathcal{H}.
\]
This property is fundamental to our analysis, as it constitutes the key ingredient in the convergence analysis of Algorithm 1. We refer the reader to \cite{LDT} for the proof and further properties of transformed resolvents.  First  we investigate a generalized forward--backward scheme based on the transformed resolvent.

\medskip

\noindent\textbf{Algorithm 1.}
\[
x_{k+1}\in(\gamma B+v)^{-1}\bigl(v(x_k)-\gamma A(x_k)\bigr), \; k\ge 0,\; x_0\in \mathcal{H}.
\]

\medskip

The following theorems  establishes the strong convergence and  linear convergence for the sequence generated by Algorithm 1.

\begin{theorem}\label{thm1} 
Suppose that Assumptions 1 and 2 hold. If $(A,v)$ is $\alpha$-strongly monotone and $A$ is $L$-Lipschitz continuous, with
\[
0<\gamma<\frac{2\alpha}{L^2},
\]
then the sequence $(x_k)$ generated by Algorithm~1 converges to the unique solution $x^*$.

Moreover:
\begin{enumerate}
\item[(a)] If $v$ is Lipschitz continuous, then $(v(x_k))$ converges linearly to $v(x^*)$.
\item[(b)] If both $v$ and $v^{-1}$ are Lipschitz continuous, then $(x_k)$ converges linearly to $x^*$.
\end{enumerate}
\end{theorem}

\begin{proof}
Since $(A,v)$ is $\alpha$-strongly monotone, the solution set $S$ has a unique point $x^*$.
{Since $v(x_{k+1})=T_{\gamma B}^v(v(x_k)-\gamma Ax_k)$ and $T_{\gamma B}^v$ is nonexpansive, we have
\baqn
\Vert v(x_{k+1})-v(x^*)\Vert^2&\le&\Vert v(x_{k})-v(x^*)-\gamma (Ax_k-Ax^*)\Vert^2\\
&=&\Vert v(x_{k})-v(x^*)\Vert^2-2\gamma\langle v(x_{k})-v(x^*)), Ax_k-Ax^*\rangle+\gamma^2 \Vert Ax_k-Ax^*\Vert^2\\
&\le&\Vert v(x_{k})-v(x^*)\Vert^2-2\gamma \alpha \Vert x_{k}-x^*\Vert^2+\gamma^2L^2\Vert x_{k}-x^*\Vert^2\\
&\le&\Vert v(x_{k})-v(x^*)\Vert^2-\beta \Vert x_{k}-x^*\Vert^2,
\eaqn
where
\[
\beta:=2\gamma\alpha-\gamma^2L^2>0.
\]
Hence, the sequence $\bigl(\|v(x_k)-v(x^*)\|\bigr)$ is monotonically decreasing and therefore convergent.  Consequently $$\|x_k-x^*\|\to0.$$
}

\noindent
 {
(a) Suppose that $v$ is Lipschitz continuous with modulus $L_v$. Then
\[
\|v(x_k)-v(x^*)\|
\le
L_v\|x_k-x^*\|.
\]
Substituting this estimate into the previous inequality gives
\[
\|v(x_{k+1})-v(x^*)\|^2
\le
\left(1-\frac{\beta}{L_v^2}\right)
\|v(x_k)-v(x^*)\|^2.
\]
Therefore, $(v(x_k))$ converges linearly to $v(x^*)$.
}
\medskip

\noindent
{(b) Assume, in addition, that $v^{-1}$ is Lipschitz continuous with modulus $\widetilde L_v$. Then
\[
\|x_{k+1}-x^*\|
\le
\widetilde L_v\|v(x_{k+1})-v(x^*)\|
\le
\widetilde L_v
\sqrt{1-\frac{\beta}{L_v^2}}
\|v(x_k)-v(x^*)\|
\le
L_v\widetilde L_v
\sqrt{1-\frac{\beta}{L_v^2}}
\|x_k-x^*\|.
\]
Hence, $(x_k)$ converges linearly to $x^*$.
}
\end{proof}


\begin{theorem}\label{thm2}
Suppose that Assumptions 1 and 2 hold. Assume that the pair $(B,v)$ is $\alpha$-strongly monotone,   the pair $(A,v)$ is $L$-Lipschitz continuous and $v$ is $L_v$-Lipschitz continuous for some $L, L_v>0$. Then the sequence $(v(x_k))$ generated by Algorithm 1 converges linearly to $v(x^*)$, provided that
\[
\gamma<\alpha L^{-2}L_v^{-2}.
\]
\end{theorem}

\begin{proof}
Since
\[
v(x_{k+1})=
T_{\gamma B}^{v}\bigl(v(x_k)-\gamma A(x_k)\bigr),
\]
and $T_{\gamma B}^{v}$ is
$\frac{1}{1+\alpha\gamma L_v^{-2}}$-Lipschitz continuous (see \cite{LDT}), it follows that
\baqn
\Vert v(x_{k+1})-v(x^*)\Vert^2&\le&\frac{1}{1+\alpha \gamma L_v^{-2}} \Vert v(x_{k})-v(x^*)-\gamma (Ax_k-Ax^*)\Vert^2\\
&=&\frac{1}{1+\alpha \gamma L_v^{-2}}(\Vert v(x_{k})-v(x^*)\Vert^2-2\gamma\langle v(x_{k})-v(x^*), Ax_k-Ax^*\rangle\\
&&\,+\gamma^2 \Vert Ax_k-Ax^*\Vert^2).
\eaqn   
Since $(A,v)$ is $L$-Lipschitz continuous, one has 
\[
\|Ax_k-Ax^*\|
\le
L\|v(x_k)-v(x^*)\|.
\]
Discarding the nonpositive cross term yields
\[
\|v(x_{k+1})-v(x^*)\|^2
\le
\frac{1+\gamma^2L^2}
{1+\alpha\gamma L_v^{-2}}
\|v(x_k)-v(x^*)\|^2=\kappa \|v(x_k)-v(x^*)\|^2
\]
where
\[
\kappa:=
\frac{1+\gamma^2L^2}
{1+\alpha\gamma L_v^{-2}}<1
\]
since
\[
\gamma<\alpha L^{-2}L_v^{-2}.
\]
The conclusion follows.
\end{proof}
%

%

\begin{theorem}\label{thm3}
Suppose that Assumptions~1 and~2 are satisfied.

\begin{enumerate}[(a)]
\item If $(A,v)$ is $\alpha$-cocoercive and $\gamma<2\alpha$, then the
sequence $(A(x_k))$ generated by Algorithm~1 converges to $A(x^*)$.
Moreover, if $A^{-1}$ is single-valued and continuous, then $(x_k)$
converges to $x^*$.

\item If $(v,A)$ is $\alpha$-cocoercive, $(A,v)$ is
$L$-Lipschitz continuous, and $\gamma<2\alpha/L^2$, then the sequence
$(v(x_k))$ generated by Algorithm~1 converges linearly to $v(x^*)$.
Moreover, if $v^{-1}$ is single-valued and Lipschitz continuous, then
$(x_k)$ converges linearly to $x^*$.
\end{enumerate}
\end{theorem}

\begin{proof}
a) Since $T_{\gamma B}^v$ is nonexpansive, one has 
\baqn
\Vert v(x_{k+1})-v(x^*)\Vert^2&\le&\Vert v(x_{k})-v(x^*)-\gamma (Ax_k-Ax^*)\Vert^2\\
&=&\Vert v(x_{k})-v(x^*)\Vert^2-2\gamma\langle v(x_{k})-v(x^*)), Ax_k-Ax^*\rangle+\gamma^2 \Vert Ax_k-Ax^*\Vert^2\\
&\le&\Vert v(x_{k})-v(x^*)\Vert^2-2\gamma \alpha \Vert Ax_k-Ax^*\Vert^2+\gamma^2\Vert Ax_k-Ax^*\Vert^2\\
&\le&\Vert v(x_{k})-v(x^*)\Vert^2-\beta \Vert Ax_k-Ax^*\Vert^2,
\eaqn
where $\beta=\gamma(2\alpha-\gamma)>0$.
Therefore $(\Vert v(x_{k+1})-v(x^*)\Vert)$ is decreasing, convergent and  $(\Vert Ax_k-Ax^*\Vert)$ converges to $0$. The continuity of  $A^{-1}$ implies the convergence of the sequence $(x_k)$ to $x^*$.\\

(b) Similarly, using the nonexpansiveness of $T_{\gamma B}^v$, we obtain
\begin{align*}
\|v(x_{k+1})-v(x^*)\|^2
&\leq
\|v(x_k)-v(x^*)\|^2
-2\gamma\langle v(x_k)-v(x^*),Ax_k-Ax^*\rangle
+\gamma^2\|Ax_k-Ax^*)\|^2\\
&\leq
\|v(x_k)-v(x^*)\|^2
-2\gamma\alpha\|v(x_k)-v(x^*)\|^2
+\gamma^2L^2\|v(x_k)-v(x^*)\|^2\\
&=
\bigl(1-\gamma(2\alpha-\gamma L^2)\bigr)
\|v(x_k)-v(x^*)\|^2\\
&=(1-\kappa)\|v(x_k)-v(x^*)\|^2,
\end{align*}
where
\[
\kappa:=\gamma(2\alpha-\gamma L^2)>0.
\]
Since $\gamma<2\alpha/L^2$, we also have $0<\kappa<1$. Hence,
\[
\|v(x_k)-v(x^*)\|
\leq
(1-\kappa)^{k/2}\|v(x_0)-v(x^*)\|,
\]
which proves the linear convergence of $(v(x_k))$ to $v(x^*)$.
If, in addition, $v^{-1}$ is single-valued and Lipschitz continuous with
Lipschitz modulus $\widetilde L_v$, then
\[
\|x_k-x^*\|
\leq
\widetilde L_v\|v(x_k)-v(x^*)\|
\leq
\widetilde L_v(1-\kappa)^{k/2}
\|v(x_0)-v(x^*)\|,
\]
and therefore $(x_k)$ converges linearly to $x^*$.
\end{proof}

Algorithm 1 relies on strong monotonicity or cocoercivity and on the full range of $v$, which may be restrictive in certain applications. The following weaker assumption allows us to construct a generalized Tseng forward–backward–forward algorithm that remains applicable under weaker conditions and provides approximate solutions to the inclusion problem.

%
%
\noindent \textbf{Assumption 3.} The mapping $v:\mathcal{H}\to\mathcal{H}$ is linear 
and $\ran  A\subset  \ran v\subset \ran (\gamma B+v)$. \\

{Under Assumption 3, the mappings $v^{-1}A$ and $J_{\gamma B}^v$ are well defined. This allows us to introduce the following Tseng-type forward-backward algorithm.}\\

\noindent {\bf Algorithm 2}.
\[
\begin{cases}
y_{k}\in (\gamma B+v)^{-1}(v(x_k)-\gamma Ax_k)=J^v_{\gamma B}(x_k-\gamma v^{-1}(Ax_k)),\\
x_{k+1}\in y_k+\gamma v^{-1}(Ax_k-Ay_k), \;k\ge 0, \;x_0\in \mathcal{H}.
\end{cases}
\]

\medskip
\begin{theorem}\label{thm4}
Suppose that Assumptions 1 and 3 hold, and that the pair $(A,v)$ is $L$-Lipschitz continuous for some $L>0$ satisfying
$$
\gamma<\frac{1}{L}.
$$
Let $(x_k)$ and $(y_k)$ be the sequences generated by Algorithm 2. Then the following assertions hold:
\begin{enumerate}
\item
$
\|v(y_k)-v(x_k)\|\to0\qquad\text{as }k\to\infty.
$

\item If $y^*$ is a weak cluster point of $(y_k)$ and the graph of $A+B$ is sequentially weak-to-strong closed, then $y^*\in S$. Moreover, if $v^{-1}:\operatorname{ran}v\to\mathcal{H}$ is Lipschitz continuous, then the sequence $(x_k)$ generated by Algorithm 2 converges weakly to a point in $S$.

\item If $(A+B)^{-1}$ is $R$-continuous at $0$, then
$
d(y_k,S)\to0.
$
\end{enumerate}
\end{theorem}

%
%
%
%
\begin{proof}
(i) Let $x^*\in S$. Then
\begin{equation}\label{discr}
\frac{v(y_k)-v(x_k)}{\gamma}+A(x_k)\in -B(y_k),
\end{equation}
and
\[
A(x^*)\in -B(x^*).
\]
Since $(B,v)$ is monotone, it follows that
\[
\left\langle
\frac{v(y_k)-v(x_k)}{\gamma}+A(x_k)-A(x^*),
v(y_k)-v(x^*)
\right\rangle
\le0.
\]
Combining this inequality with the monotonicity of $(A,v)$, we obtain
\begin{equation}\label{eq0}
\left\langle
\frac{v(y_k)-v(x_k)}{\gamma}+A(x_k)-A(y_k)
,v(y_k)-v(x^*)
\right\rangle
\le0.
\end{equation}
On the other hand, Algorithm~2 yields
\begin{equation}\label{eq00}
A(x_k)-A(y_k)=
\frac{1}{\gamma}\bigl(v(x_{k+1})-v(y_k)\bigr).
\end{equation}
Substituting \eqref{eq00} into \eqref{eq0} and using $\gamma>0$, we obtain
\[
\langle
v(x_{k+1})-v(x_k)
,v(y_k)-v(x^*)
\rangle
\le0.
\]
Hence,
\baqn
	&&\langle v(x_{k+1})-v(x^*), v(y_k)-v(x^*)\rangle\le \langle v(x_{k})-v(x^*),\, v(y_k)-v(x^*)\rangle\\
	&\Leftrightarrow& \Vert v(x_{k+1})-v(x^*)\Vert^2+\Vert v(y_k)-v(x^*)\Vert^2-\Vert v(x_{k+1})-v(y_k)\Vert^2\\
	&\le&\Vert v(x_{k})-v(x^*)\Vert^2+\Vert v(y_k)-v(x^*)\Vert^2-\Vert v(x_{k})-v(y_k)\Vert^2\\
	&\Leftrightarrow& \Vert v(x_{k+1})-v(x^*)\Vert^2\le \Vert v(x_{k})-v(x^*)\Vert^2+\gamma^2 \Vert Ax_k-Ay_k\Vert^2-\Vert v(x_{k})-v(y_k)\Vert^2,
	\eaqn
where the last equality follows from \eqref{eq00}.

Since $(A,v)$ is $L$-Lipschitz continuous and
$\gamma^2L^2<1,
$ there exists $\varepsilon>0$ such that
$
\gamma^2L^2\le1-\varepsilon.
$
Therefore,
\[
\|v(x_{k+1})-v(x^*)\|^2
\le
\|v(x_k)-v(x^*)\|^2
-\varepsilon\|v(x_k)-v(y_k)\|^2.
\]
It follows that the sequence
$(\Vert v(x_{k+1})-v(x^*)\Vert)$  is nonincreasing and hence convergent.  Moreover,
\[
\|v(x_k)-v(y_k)\|\longrightarrow0.
\]

\medskip

\noindent
(ii) Let $y^*$ be a weak cluster point of $(y_k)$. By \eqref{discr},
\begin{equation}\label{discr1}
\frac{v(y_k)-v(x_k)}{\gamma}
+A(x_k)-A(y_k)
\in
-(A+B)(y_k).
\end{equation}
Since
\[
\|v(x_k)-v(y_k)\|\to0,
\]
and $(A,v)$ is $L$-Lipschitz continuous, it follows that
\[
\frac{v(y_k)-v(x_k)}{\gamma}
+A(x_k)-A(y_k)
\longrightarrow0.
\]
Using the assumption that the graph of $A+B$ is sequentially weak-to-strong closed, we conclude that
\begin{equation} \label{belong}
0\in A(y^*)+B(y^*),
\end{equation}
that is,
$
y^*\in S.
$

\medskip
Now suppose that $v^{-1}:\operatorname{ran}v\to\mathcal{H}$ is Lipschitz continuous. Since the sequence
$$
\bigl(\|v(x_{k+1})-v(x^*)\|\bigr)
$$
converges and
$$
\|v(x_k)-v(y_k)\|\to0,
$$
the Lipschitz continuity of $v^{-1}$ implies that
$
\bigl(\|x_{k+1}-x^*\|\bigr)
$
also converges and that
$
\|x_k-y_k\| \to 0.
$
Let $x^*$ be a weak cluster point of $(x_k)$. Since $\|x_k-y_k\|\to0$, $x^*$ is also a weak cluster point of $(y_k)$. By \eqref{belong}, it follows that $x^*\in S$. Finally, Opial's lemma implies that $(x_k)$ converges weakly to some point in $S$.

\noindent
(iii) Define
\[
u_k:=
\frac{v(y_k)-v(x_k)}{\gamma}
+A(x_k)-A(y_k).
\]
Then $u_k\to0$. By \eqref{discr1},
\[
y_k\in(A+B)^{-1}(-u_k).
\]
Since $(A+B)^{-1}$ is $R$-continuous at $0$, we have
\[
(A+B)^{-1}(-u_k)
\subset
(A+B)^{-1}(0)+\rho(\|u_k\|),
\]
where $\rho(t)\to0$ as $t\to0$. Consequently,
\[
y_k
\in
S+\rho(\|u_k\|),
\]
which implies that
\[
d(y_k,S)\longrightarrow0.
\]
This completes the proof.
\end{proof}

\begin{remark}
	\label{rem:3.5}
	\rm(a) At the end of (ii) in Theorem \ref{thm4}, if
	$v^{-1}:\operatorname{ran}v\to \mathcal{H}$ is not single-valued and Lipschitz continuous, {but there exists a single-valued Lipschitz continuous selection $\widetilde v^{-1}$} , then $v^{-1}$ can be replaced by $\widetilde v^{-1}$ and the same conclusion follows.
	
	\rm(b) If the range condition
	\[
	\operatorname{ran} A\subset \operatorname{ran}v
	\subset \operatorname{ran}(\gamma B+v)
	\]
	is not satisfied, one may consider an alternative decomposition
	\[
	A+B=A'+B'
	\]
	such that, for a suitable choice of $v$, the corresponding range condition
	\[
	\operatorname{ran} A'\subset \operatorname{ran}v
	\subset \operatorname{ran}(\gamma B'+v)
	\]
	holds. This provides additional flexibility in applying the generalized framework.
\end{remark}

\section{Numerical Examples}\label{sec4}

\begin{example}
First we consider the operators $A: \R^3\to \R^3, B: \R^3 \rightrightarrows \R^3$ defined by
$$
A(x)=Qx+b, \qquad B(x)= N_{\R^3_+}(x),
$$
where
$$
Q =\begin{bmatrix} 1 & -3 & 0 \\ 0 & 2 & 0 \\ 0 & 0 & 1 \end{bmatrix},  \quad
b = \begin{bmatrix} -1 \\ 1 \\ -1 \end{bmatrix}.
$$
\medskip
We choose 
$$
v=\begin{bmatrix} 1 & 0 & 0\\ 0 & 3 & 0 \\ 0 & 0 & 3\end{bmatrix}.
$$
Then $A$ is non-monotone,  $v$ is invertible, $(B,v)$ is monotone, $(A,v)$ is $\alpha$-strongly monotone with $\alpha=\frac{7-\sqrt{34}}{2}\approx0.58$ and $A$ is $L$-Lipschitz continuous with {$L\approx3.71.$} Moreover, both $v$ and $v^{-1}$ are Lipschitz continuous. Hence all assumptions of Theorem \ref{thm1} are satisfied, and the convergence result follows. 
A direct computation yields the unique solution $x^*= (1,0,1)$.
\medskip

\noindent

To illustrate the convergence behavior of Algorithm 1, we compute the errors $\|x_k-x^*\|_2$ and $\|v(x_k)-v(x^*)\|_2$. Starting from $x_0=(4,3,5)$ and using the admissible stepsize $\gamma=0.06$,  the corresponding convergence behavior is displayed in Figure \ref{error1}.
\begin{figure}[H]
	\centering{\includegraphics[scale=0.35]{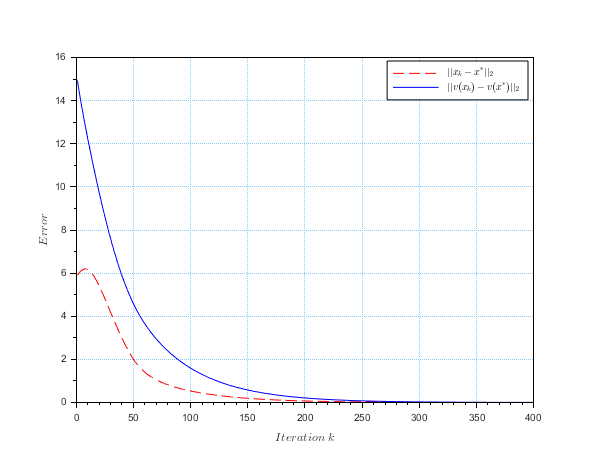}}
	\caption{\centering {\small Iterate errors $\|x_k-x^*\|_2$ and $\|v(x_k)-v(x^*)\|_2$ for $\gamma=0.06$ and $x_0=(4,3,5)$.}}
	\label{error1}
\end{figure}

Figure \ref{error2} further illustrates the effect of the stepsize and the initial point. In the left panel, starting from $x_0=(4,3,5)$, we plot the error $\|x_k-x^*\|_2$ for different values of $\gamma$, while right panel shows its evolution for $\gamma=0.06$ and various initial points. 

\begin{figure}[H]
	\centering{\includegraphics[scale=0.35]{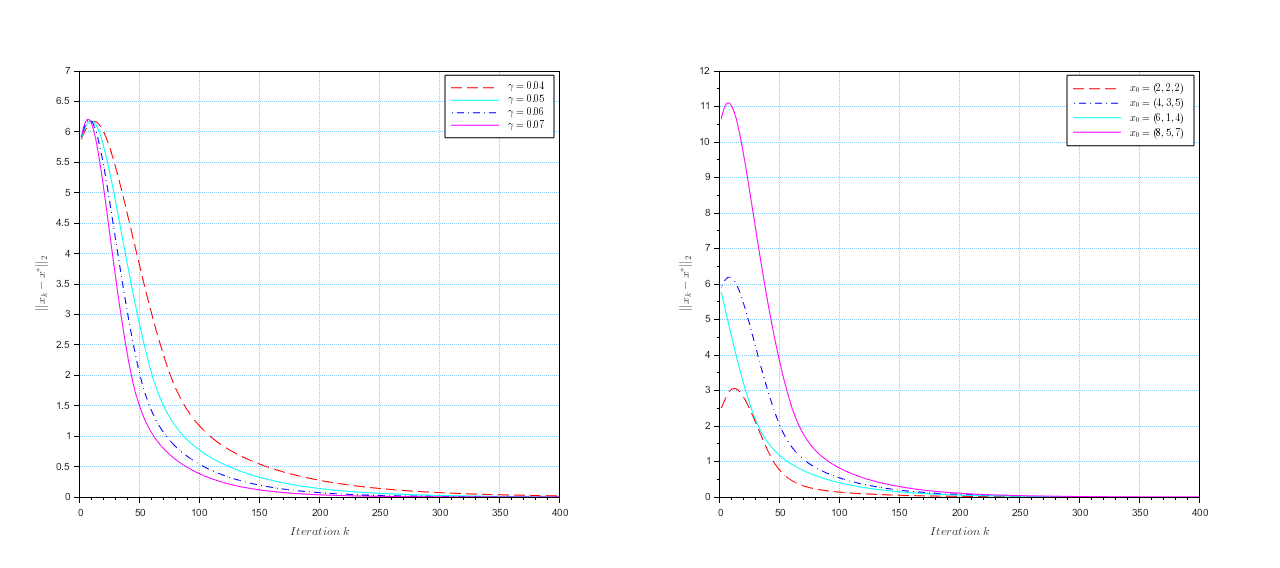}}
	\caption{\centering {\small Iterate error $\|x_k-x^*\|_2$ for various values of $\gamma$ and initial points.}}
	\label{error2}
\end{figure}
\end{example}

\begin{example}
Next we consider the operators $A: \R^3\to \R^3, B: \R^3 \rightrightarrows \R^3$ defined by
$$
A(x)=Qx+b, \qquad B(x)= N_{\R_+}(x_1)\times N_{\R_+}(x_2)\times \{N_{\R_+}(x_3)+x_3+2x_1-3x_2\},
$$
where
$$
Q =\begin{bmatrix} 1 & -3 & 0 \\ 0 & 2 & 0 \\ 0 & 0 & 0 \end{bmatrix},  \quad
b = \begin{bmatrix} -1 \\ 1 \\ 0 \end{bmatrix}.
$$
We choose 
$$
v=\begin{bmatrix} 1 & 0 & 0\\ 0 & 3 & 0 \\ 0 & 0 & 0\end{bmatrix}.
$$
Then $A$ and $B$ are non-monotone but  $(A,v)$ and $(B,v)$ are monotone and 
$$
\ran  A = {\R^2\times \{0\}= \ran v}\subset \ran (\gamma B+v)=\R^3.
$$
{Moreover $(A,v)$ is $L$-Lipschitz continuous with $L\approx 1.5$.}  Moreover, $A$ is continuous and $B$ has a closed graph, which implies that $\operatorname{gra}(A+B)$ is sequentially weak-to-strong closed. In addition,
$
S=(A+B)^{-1}(0)=\{x^*\},
\;
x^*=(1,0,0),
$
and {$(A+B)^{-1}$ is $R$-Lipschitz continuous at $0$ with Lipschitz constant $1$ and radius $\sigma=1/2$.}
Thus all assumptions of Theorem \ref{thm4} are satisfied.

Starting from the initial point $x_0=(2,1,1),$ the sequence generated by Algorithm 2 converges to $x^*$. Figure \ref{error3} illustrate this convergence for $\gamma=0.1,0.2,0.4,$ and $0.6$, all satisfying $\gamma<1/L$. 

\begin{figure}[H]
	\centering{\includegraphics[scale=0.35]{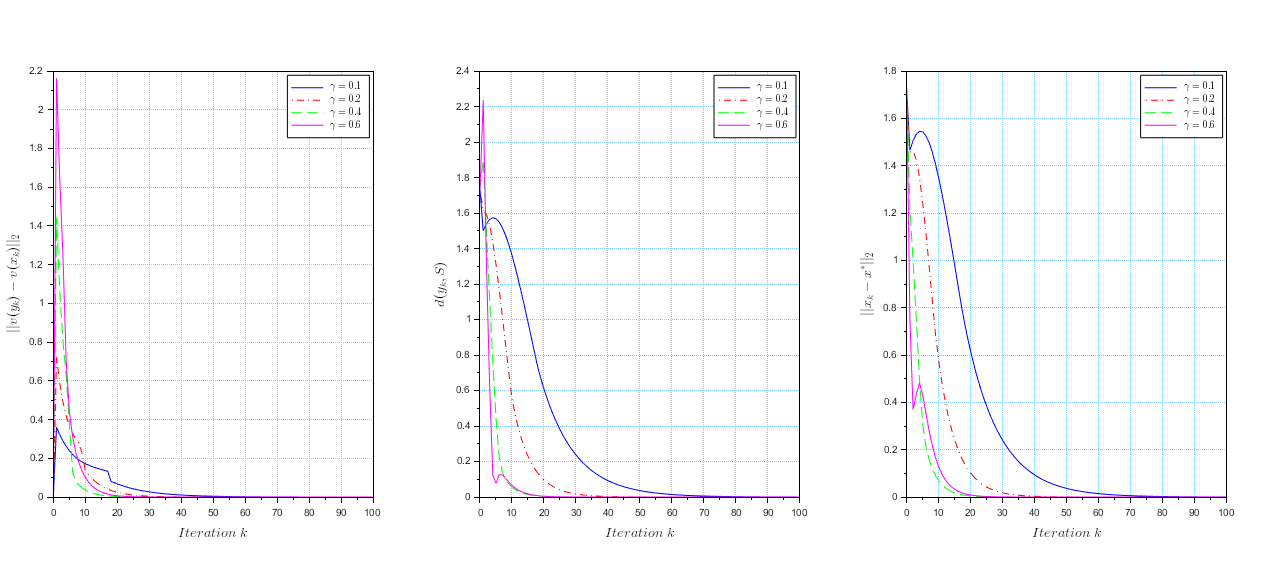}}
	\caption{\centering {\small Iterate errors $\|v(y_k)-v(x_k)\|_2$, $d(y_k,S)$,  $\|x_k-x^*\|_2$ for different values of $\gamma$ and initial points.}}
	\label{error3}
\end{figure}

The panels show that the quantities $\|v(y_k)-v(x_k)\|_2$, $d(y_k,S)$, and $\|x_k-x^*\|_2$ all converge to zero, confirming numerically the convergence established by Theorem \ref{thm4}. Moreover, larger admissible values of $\gamma$ lead to faster numerical convergence.
\end{example}


%
%
%

\section{Conclusion}\label{sec5}
{This work extends the recently introduced framework of pair monotonicity from nonmonotone single-operator inclusion problems to structured sum inclusion problems involving a single-valued operator and a set-valued operator. It demonstrates that pair monotonicity provides a suitable framework for the development of generalized  splitting methods beyond the classical monotone setting. 
Two algorithms based on the transformed and warped resolvents were proposed and analyzed. Weak, strong and linear convergence results were established under suitable assumptions. Numerical examples confirmed the theoretical convergence properties and illustrated the behavior of the proposed methods.
The present work opens several perspectives for future research. These include inertial and accelerated variants and applications to non-convex constrained optimization.}

\end{document}